\documentclass[10pt,A4paper,fleqn]{amsart}
\usepackage{mathrsfs}
\usepackage{amsfonts}
\usepackage{txfonts}
\usepackage{amsmath}
\usepackage{amssymb}
\usepackage{amsthm}
\usepackage[pdftex]{graphicx}
\usepackage[toc,page,title,titletoc,header]{appendix}
\usepackage{geometry}
\usepackage{upgreek}
\usepackage[T1]{fontenc}
\usepackage{color}
\usepackage{hyperref}
\usepackage[all]{xy}
\usepackage[french,english]{babel}

\theoremstyle{plain}
\newtheorem{theo}{Theorem}[section]
\newtheorem*{theo*}{Theorem}
\newtheorem{prop}{Proposition}[section]

\newtheorem*{lem*}{Lemma}

\theoremstyle{definition}

\theoremstyle{remark}

\newtheorem*{ack}{Acknowledgements}
\newtheorem*{rem*}{Remark}
\newtheorem{rem}{Remark}[section]

\newtheorem{claim}{Claim}[section]

\newcommand{\R}{\mathbb{R}}

\begin{document}

\bibliographystyle{unsrt}

\abstract Based on a conceptual link between the partial reduction procedure of the reduction of the rotational symmetry of the N-body problem with the symplectic cross-section theorem of Guillemin-Sternberg, we present alternative proofs of the symplecticity of Delaunay and Deprit coordinates in celestial mechanics.
\endabstract

\title{Partial Reduction and Delaunay/Deprit Variables}

\author{Lei Zhao}
\address{Johann Bernoulli Institute for Mathematics and Computer Science, University of Groningen, The Netherlands}
\email{l.zhao@rug.nl}


\date\today
\maketitle

\tableofcontents

\section{Introduction}
In this article, we aim to present a conceptual link between the idea of the \emph{partial reduction} procedure \cite{Malige} in the reduction of the SO(3)-symmetry of the three-body and $N$-body problems (whose phase spaces, after reduction by the translation symmetries, are denoted indifferently by $\Pi$) with the symplectic cross-section theorem of Guillemin-Sternberg, and present its role in the deduction of several important Darboux coordinates, the Delaunay and the Deprit coordinates of celestial mechanics.

Following Jacobi, we know that the (full) reduction of the SO(3)-symmetry can be achieved, by example, by fixing the total angular momentum $\vec{C}$ of the system and rule out the SO(2)-symmetry of rotations around the direction of $\vec{C}$. The method of partial reduction proposed in \cite{Malige} is to only fixing the direction of the angular momentum. The resulting submanifold of the phase space is symplectic, and the restriction of the SO(3)-symmetry in this submanifold becomes the symmetry of SO(2), a maximal torus of SO(3). We see that the action of the maximal torus SO(2) has a non-trivial dynamical effect (a periodic orbit in the $SO(3)$-reduced system are in general only quasi-periodic in the  system with the $SO(3)$-symmetry, with an addition frequency corresponds to the action of the maximal torus), while the rotation of the direction of the total angular momentum only send orbits to other orbits do not interfere the essence of the dynamics. We remark that more generally, this procedure can be achieved for a Hamiltonian action of an arbitrary compact connected Lie group $Gr$ on a symplectic manifold $(M, \omega)$ with moment map $\mu$. In this general context, the partial reduction procedure is achieved by fixing a Cartan subalgebra $\mathfrak{h}^{*}$ in $\mathfrak{g}^{*}$ (where $\mathfrak{g}$ denotes the Lie algebra of $Gr$), fixing a Weyl chamber $\mathfrak{t}_{+}^{*}$ in $\mathfrak{h}^{*}$ and consider the set $\mu^{-1} (\mathfrak{t}_{+}^{*})$. A theorem of Guillemin and Sternberg states that this set is a symplectic manifold. The restriction of $G_r$-action to $\mu^{-1} (W_{+})$ (which is called a symplectic cross-section of the $G_r$ action) is thus the Hamiltonian action of one of its maximal torus. (The book \cite{FultonHarris} provides a nice presentation of all the involved notions in the theory of Lie groups and Lie algebra.) Being abelian, a torus symmetric group is in general much easier to handle. 

With the help of this construction, we shall deal with some concrete problem of determining action-angle coordinates for $N-1$ uncoupled Keplerian ellipses. A generic SO(3)-coadjoint orbit is homeomorphic to $S^{2}$, which only admits one invariant symplectic form up to multiplication of a constant. By determining this constant in concrete circumstances, we can thus recover the symplectic form on $\Pi$ from its restriction to the symplectic cross section and the Kirillov-Konstant symplectic form of the coadjoint orbits. In such a way, we obtain alternative proofs of the symplecticity of the important Delaunay and Deprit coordinates in celestial mechanics, avoiding the use of Hamilton-Jacobi methods.

We organize this article as the following: In Section \ref{Sec: 2}, we recall the Hamiltonian formulation of the three-body problem and the reduction of the translation-invariance using the Jacobi coordinates. In Section \ref{Sec 3}, we recall the reduction of the rotation-invariance of Jacobi and Deprit.  In Section \ref{Partial Reduction}, we indicate the link of these reduction procedures with the symplectic cross-section theorem of Guillemin-Sternberg. In Section \ref{Sec 5} we prove a theorem on the form of the complementary part of the symplectic form, which is then applied in Section {Symplecticity of Delaunay and Deprit coordinates} to (re-)establish the symplecticity of the Delaunay and Deprit coordinates.

\section{The Three-body Problem and the Jacobi Decomposition}\label{Sec: 2}

The three-body problem is a Hamiltonian system with phase space
$$\left\{(p_{j}, q_{j})_{j=0,1,2}=(p_{j}^{1}, p_{j}^{2}, p_{j}^{3}, q_{j}^{1}, q_{j}^{2}, q_{j}^{3}) \in (\R^{3} \times \R^3)^3 |\,  \forall 0 \leq j \neq k \leq 2, q_j \neq q_k \right\}, $$
(standard) symplectic form 
$$\omega_{0}=\sum^{2}_{j=0} \sum^{3}_{l=1} d p_j^l \wedge d q_j^l,$$
and the Hamiltonian function
$$F=\dfrac{1}{2} \sum_{0 \le j \le 2} \dfrac{\|p_j\|^2}{m_j} -  \sum_{0 \le j < k \le 2} \dfrac{m_j m_k}{\|q_j- q_k\|},$$
in which $q_0,q_1,q_2$ denote the positions of the three particles, and $p_0,p_1,p_2$ denote their conjugate momenta respectively. The physical space $\R^{3}$ is equipped with the usual Euclidean norm $\|\cdot\|$. The gravitational constant has been set to $1$.

The Hamiltonian $F$ is invariant under translations in positions. To symplectically reduce the system by this symmetry, we may switch to the \emph{Jacobi (baricentric) coordinates} $(P_i, Q_i),i=0, 1, 2,$ with
\begin{equation*} 
\left\{
\begin{array}{l} P_0=p_0+p_1+ p_2 \\ P_1=p_1+ \sigma_1 p_2\\ P_2 = p_2
\end{array} \right.
\hbox{ \phantom{aaaaaaaqqqqaa}}
\left\{
\begin{array}{l} Q_0=q_0 \\ Q_1=q_1- q_0 \\ Q_2=q_2-\sigma_0 q_0-\sigma_1 q_1,
\end{array} \right.
\end{equation*}
in which
$$\dfrac{1}{\sigma_0}=1+\dfrac{m_1}{m_0},  \dfrac{1}{\sigma_1}=1+\dfrac{m_0}{m_1}.$$
The Hamiltonian $F$ is thus independent of $Q_{0}$ due to the symmetry. We fix $P_{0}=0$ and reduce the translation symmetry by eliminating $Q_{0}$. 
In the (reduced) coordinates $(P_i, Q_i),i=1, 2$, the function $F=F(P_{1}, Q_{1}, P_{2}, Q_{2})$ describes the motions of two fictitious particles.

In the same fashion (c.f. \cite[n.385]{Wintner}), we may reduce the translation symmetry of the N-body problem, and to study the (reduced) dynamics of $N-1$ fictitious particles.

\section{Reductions: from Jacobi to Deprit}\label{Sec 3}
The group SO(3) acts on $\Pi$, the reduced phase space of the three-body problem by the translation symmetry, by simultaneously rotating the two relative positions $Q_{1}, Q_{2}$ and the two relative momenta $P_{1}, P_{2}$. This action is Hamiltonian under the standard symplectic form on $\Pi$, the Hamiltonian $F$ is invariant under this SO(3)-action, and its moment map is the total angular momentum\footnote{We have identified $\mathfrak{so}^{*}(3)$, the space of $3 \times 3$ anti-symmetric matrices, with $\R^{3}$ in the standard way.} $\vec{C}=\vec{C_1}+\vec{C_2}$, in which $\vec{C}_{1}:=Q_1 \times P_1$ and $\vec{C}_{2}:=Q_2 \times P_2$\label{Not: Ang momenta}. The reduction procedure can then be achieved by fixing the moment map $\vec{C}$ (equivalently, the direction of $\vec{C}$ and $C=|\vec{C}|$) to a regular value (i.e. $\vec{C} \neq \vec{0}$) and then reducing the system from the SO(2)-symmetry around $\vec{C}$. As SO(3) also acts on the space of (oriented) directions of $\vec{C}$, the reduced system one obtains must be independent of the direction of $\vec{C}$, and therefore has 4 degrees of freedom.

The plane perpendicular to the total angular momentum $\vec{C}$ is invariant. It is called the \emph{Laplace plane}. Choosing the Laplace plane as the reference plane\footnote{i.e. the horizontal plane.} (i.e. fix $\vec{C}$ vertical) shall give us a very convenient way of calculating the reduced Hamiltonian, as was obtained by Jacobi. Nevertheless, we can also fix $\vec{C}$ non-orthogonal to the reference plane. In this case, the Deprit coordinates shall provide us an explicit reduction procedure.

\subsection{Jacobi's elimination of the nodes of the three-body problem}

As the angular momenta $\vec{C_1}$, $\vec{C_2}$ of the two Keplerian motions and the total angular momentum $\vec{C}=\vec{C_1}+\vec{C_2}$ must lie in the same plane, the node lines of the Laplace plane with the orbital planes of the two ellipses must coincide. 

We now describe the two Keplerian motions in Delaunay variables. Let $a_1,a_2$ be the semi major axes of the inner and outer ellipses respectively. 

The Delaunay coordinates 
$$(L_i,l_i,G_i,g_i,H_i,h_i),i=1,2$$ for both ellipses are thus defined as:

\begin{equation*} 
\left\{
\begin{array}{ll}L_i=\mu_i \sqrt{M_i} \sqrt{a_i}   & \hbox{circular angular momentum}\\ l_i  &\hbox{mean anomaly}\\ G_i = L_i \sqrt{1-e_i^2} &\hbox{angular momentum} \\g_i &\hbox{argument of pericentre} \\ H_i=G_i \cos i_i &\hbox{vertical component of the angular momentum} \\ h_i &\hbox{ longitude of the ascending node},
\end{array}\right.
\end{equation*}

\begin{figure}
\centering
\includegraphics[width=3in]{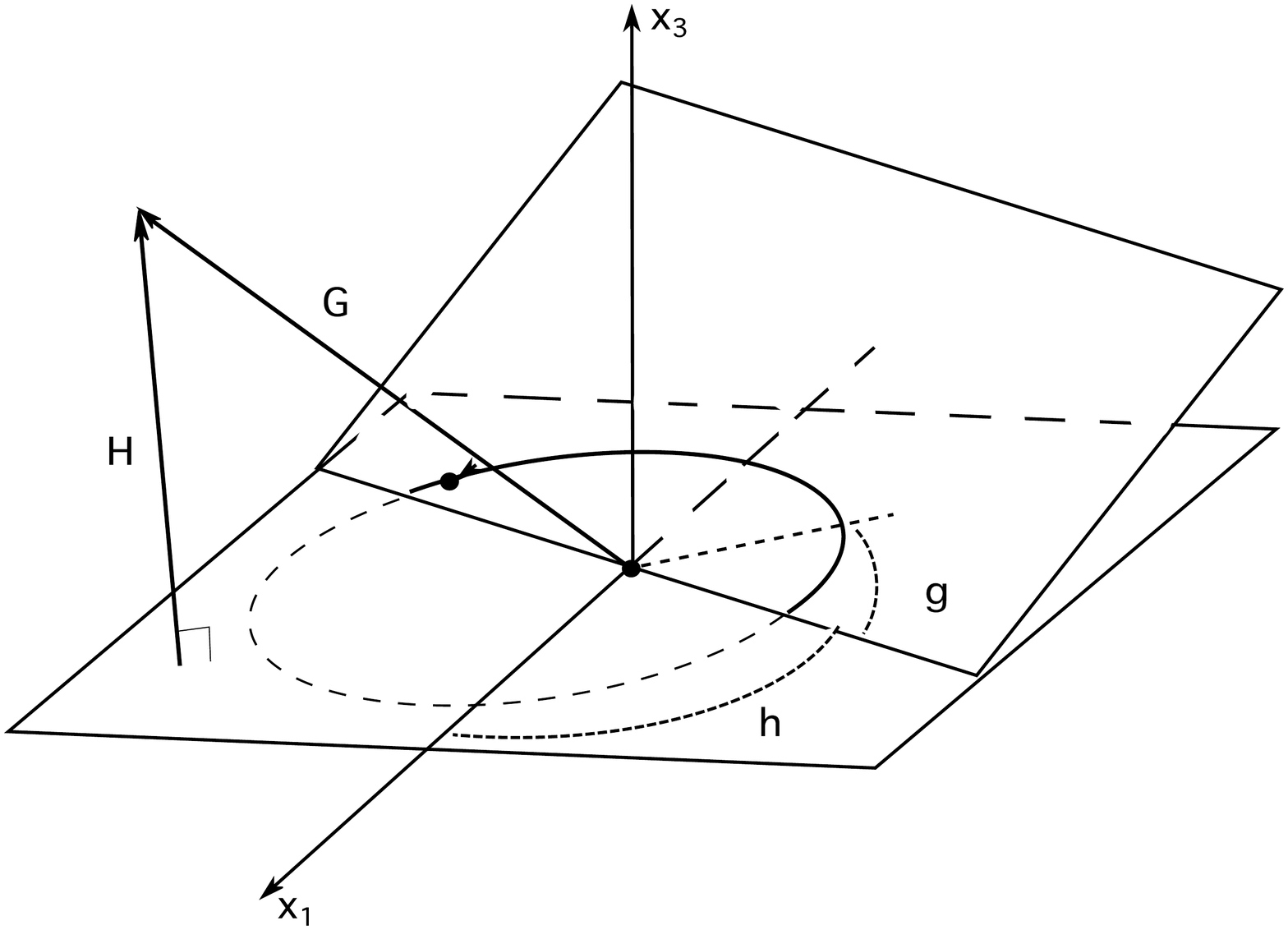}
\caption{Some Delaunay Variables}
\end{figure}
in which $e_{1}$, $e_{2}$ are the eccentricities and $i_1, i_2$ are the inclinations of the two ellipses respectively. We shall write $(L, l, G, g, H, h)$ to denote the Delaunay coordinates for a body moving on an general Keplerian elliptic orbit. From their definitions, we see that these coordinates are well-defined only when neither of the ellipses is circular, horizontal or rectilinear. We refer to \cite{Lecons}, \cite{DelaunayPoincare} or \cite[appendix A]{FejozHabilitation} for more detailed discussions of Delaunay coordinates.

By choosing the Laplace plane as the reference plane, we can express $H_1,H_2$ as functions of $G_1$, $G_2$ and $C:=\|\vec{C}\|$ as:
$$H_{1}=\dfrac{C^{2}+G_{1}^{2}-G_{2}^{2}}{2 C}, H_{2}=\dfrac{C^{2}+G_{2}^{2}-G_{1}^{2}}{2 C}.$$
Since $\vec{C}$ is vertical, we have $d H_{1} \wedge d h_{1}+ d H_{2} \wedge d h_{2}=d C \wedge d h_{1}$.
We can then reduce the system by the $\hbox{SO}(2)$-symmetry around the direction of $\vec{C}$. The degrees of freedom of the system is then reduced from 6 to 4.
  
This reduction procedure was first carried out by Jacobi and is thus called ``Jacobi's elimination of the nodes''. 

Denote by $\Pi'_{vert}$ \label{Not: Pi'_{vert}} the subspace of $\Pi$ one gets by posing $C \neq 0$ and fix the direction of $\vec{C}$ to the vertical direction $(0,0,1)$. The space $\Pi'_{vert}$ is an invariant symplectic submanifold of $\Pi$. Jacobi's elimination of node implies that the coordinates 
$$(L_{1}, l_{1}, G_{1}, g_{1}, L_{2}, l_{2}, G_{2}, g_{2}, C, h_{1})$$
are Darboux coordinates on a dense open set\footnote{on which all the variables are well-defined, i.e. the ellipse they describe are non-degenerate, non-circular, non-horizontal.} of $\Pi'_{vert}$. 

\subsection{Reduction of the three-body problem in the Deprit variables}\label{Subsection: Deprit 3BP}
Let us consider an invariant submanifold $\Pi'$ \label{Not: Pi'} of $\Pi$ by properly fixing the direction of $\vec{C} \neq 0$. The dense open set of $\Pi$ with non-vanishing $\vec{C}$ is thus the union of such invariant symplectic manifolds, and any two of them can be transformed between them by a rotation. In $\Pi'$, the SO(3)-symmetry of the system $F$ is restricted to a (Hamiltonian) SO(2)-symmetry, and is easier to handle. As the standard symplectic form on $\Pi$ is invariant under the SO(3)-action, $\Pi'$ is also an invariant symplectic submanifold of $\Pi$. We can now choose restrict the dynamical study of $F$ to $\Pi'$. Following \cite{Malige}, this restriction procedure is called \emph{partial reduction}. 

For $\vec{C}$ non-vertical, the reduction procedure is conveniently understood in the \emph{Deprit coordinates}\footnote{The terminology follows from \cite{ChierchiaPinzari}.} 
$$(L_1,l_1,L_2,l_2,G_1,\bar{g}_1,G_2,\bar{g}_2,\Phi_1,\varphi_1,\Phi_2,\varphi_2),$$\label{Not: Deprit coordinates}defined as follows: Let   $\nu_L$ be the intersection line of the two orbital planes\footnote{This is the common node line of the two planes in the Laplace plane.}, $\nu_{T}$ be the intersection of the Laplace plane with the horizontal reference plane. {We orient $\nu_{L}$ by the ascending node of the inner ellipse, and choose any orientation for  $\nu_{T}$.} Let
\begin{itemize}
  \item $\bar{g}_1,\bar{g}_2$ denote the angles from $\nu_{L}$\footnote{ A conventional choice of orientation of the node line, is given by their ascending nodes, which leads to opposite orientations of $\nu_{L}$ in the definition of $\bar{g}_{1}$ and $\bar{g}_{2}$. } to the pericentres;
  \item $\varphi_1$ denotes the angle from $\nu_{T}$\label{Not: node lines} to $\nu_L$;
  \item $\varphi_2$ denotes the angle from the first coordinate axis in the reference plane to $\nu_T$;
  \item $\Phi_1=C=\|\vec{C}\|$,  $\Phi_2=C_z=\hbox{the vertical component of }\vec{C}$\label{Not: total ang momentum and vertical component}.
\end{itemize}

\begin{figure}
\centering
\includegraphics[width=4in]{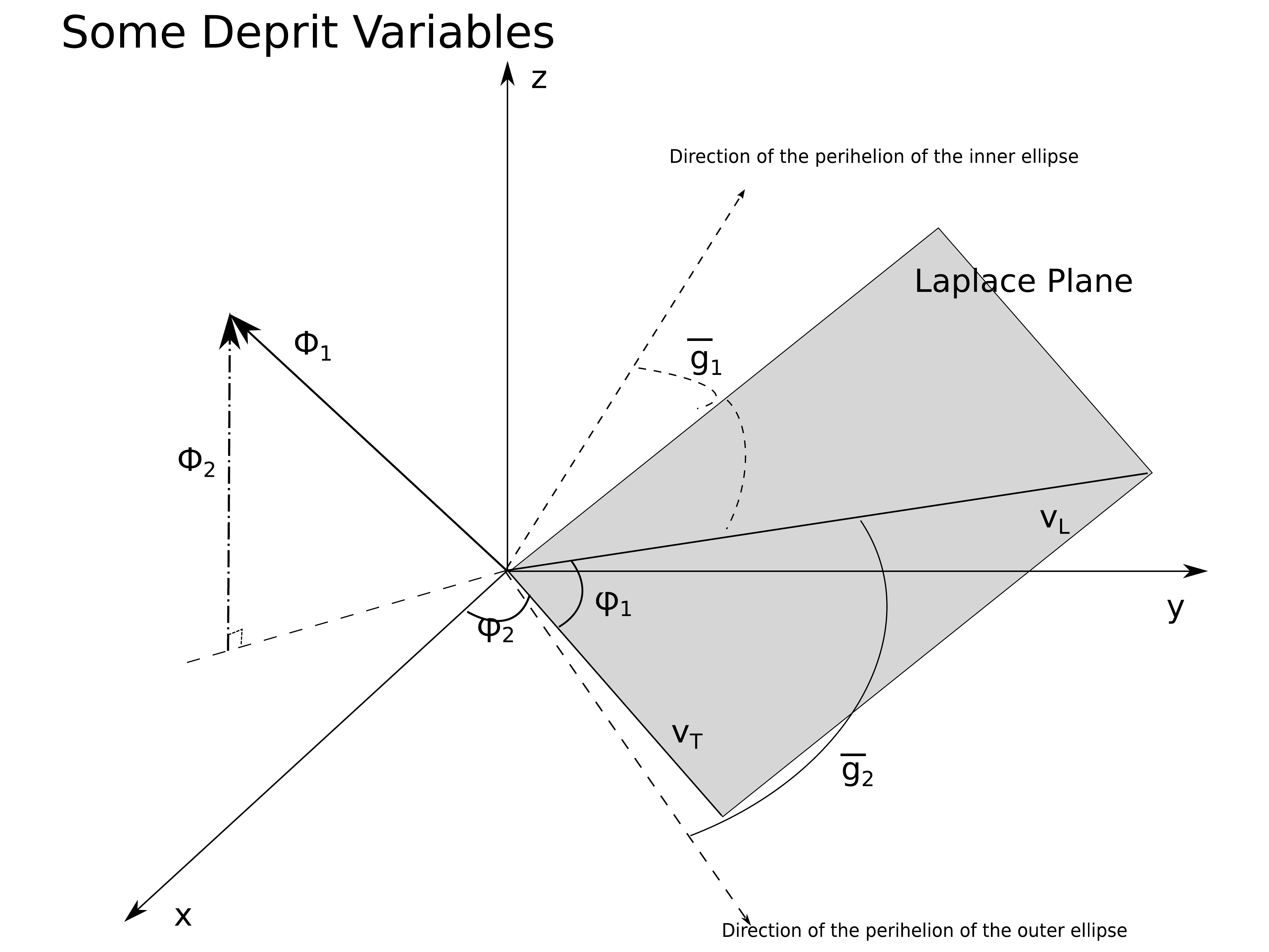}
\caption{Some Deprit Variables} \label{DepritCoordinates}
\end{figure}

\begin{prop}(Chierchia-Pinzari \cite{ChierchiaPinzari})
Deprit coordinates are Darboux coordinates. In the open dense subset of $\Pi$ where all the Deprit variables are well-defined, we have:
\small
\begin{equation*}
\begin{split}
\omega_0 =d L_1 \wedge d l_1 + d G_1 \wedge d  \bar{g}_1 + d L_2  \wedge d l_2 +  d G_2 \wedge d \bar{g}_2 +  d \Phi_1 \wedge d \phi_1 +  d \Phi_2 \wedge d \phi_2.
\end{split}
\end{equation*}
\normalsize
\end{prop}

The variables $(L_1,l_1,L_2,l_2,G_1,\bar{g}_1,G_2,\bar{g}_2,\Phi_1,\varphi_1)$ form a set of Darboux coordinates on a dense open set (on which all the variables are well-defined) of $\Pi'$, any of the subspaces of $\Pi$ one gets by fixing the direction of $\vec{C}$ non-vertical. In these coordinates, the Hamiltonian can be written in closed form in the ``planar variables'' $(L_1,l_1,G_1,\bar{g}_1,L_2,l_2,G_2,\bar{g}_2)$ and $C$. We can then fix $C$ and reduce the system from the SO(2)-symmetry around the direction of $\vec{C}$ to complete the reduction procedure.

In \cite{Deprit}, Deprit established a set of coordinates closely related to the set of coordinates presented above. The actual form of our Deprit coordinates was independently discovered and first presented by Chierchia and Pinzari in \cite{ChierchiaPinzari}. Note that in both of these references, Deprit coordinates are built for the general N-body problem, with the aim to generalize Jacobi's elimination of nodes, or to conveniently reduce the SO(3)-symmetry of the $N$-body problem for $N \ge 4$, which is of significant importance for the perturbative study of the $N$-body problem (c.f.\cite{ChierchiaPinzariPlanetary}). 

\begin{rem} In $\Pi'_{vert}$, we have $\bar{g}_1=g_1$, $\bar{g}_2=g_2$ and $\Phi_1=\Phi_2$. The angles $\phi_1$, $\phi_2$ are not defined individually. Nevertheless, their sum $\phi_1+ \phi_2$ remains well defined. One can then recover Jacobi's elimination of the node from the Deprit variables by a limit procedure, see \cite{ChierchiaPinzari} for details. 
\end{rem}

\subsection{Deprit coordinates for $N$-body problem}
Let us present the Deprit coordinates in $N$-body problem, or for $N-1$ Keplerian ellipses\footnote{The reader understands that more precisely this means Keplerian elliptic motions.}, by induction on $N$: Divide the $N-1$ Keplerian ellipses into a group of $N-2$ Keplerian ellipses and another group consists of only one Keplerian ellipse (whose elements are written with an subscript $N-1$). Denote the total angular momentum of the $N-2$ Keplerian ellipses in the first group by $\vec{C}_{N-2}$ and the total angular momentum of the whole system by $\vec{C}$. Then the Deprit coordinates for the group of $N-2$ Keplerian ellipses, except the conjugate pair of $C_{N-2, z}$ and its conjugate angle, together with $L_{N-2}, l_{N-2}, G_{N-2}, \bar{g}_{N-2}, C, \phi_{1}, C_{z}, \phi_{2}$ are the Deprit coordinates for the $N-1$ Keplerian ellipses, in which $C_{N-2, z}$ is the projection of $\vec{C}_{N-2}$ to $\vec{C}$, and $\bar{g}_{N-2}$ is the argument of the perihelion from the node line of this Keplerian ellipse with the Laplace plane, \emph{i.e.} the plane orthogonal to $\vec{C}$. More explicit and precise definitions of these variables can be found in \cite{ChierchiaPinzari}.

As mentioned above, the symplecticity of these set of coordinates is proven by Chierchia and Pinzari \cite{ChierchiaPinzari} (Deprit proved the symplecticity of his set of coordinates in \cite{Deprit}). We shall give an alternative proof in Section \ref{Symplecticity of Delaunay and Deprit coordinates}.

\section{A Conceptual View of the Partial Reduction Procedure}\label{Partial Reduction}
Now let us make an remark on the generalization of the idea of partial reduction \cite{Malige} for arbitrary compact connected group $G_r$, which simultaneously gives a conceptual way of understanding this procedure. 

Let $G_r$ be a compact connected Lie group which acts in a Hamiltonian way on a connected symplectic manifold $(M, \omega)$ and let $\mu: M \to \mathfrak{g}^{*}$ be the associated moment map, in which $\mathfrak{g}^{*}$ is the dual of the Lie algebra $\mathfrak{g}$ of $G_r$. Since $G_{r}$ is compact, there exist an invariant inner product on $\mathfrak{g}$, which permits to identify $\mathfrak{g}$ with its dual $\mathfrak{g}^{*}$. For any fixed Cartan subalgebra $\mathfrak{h} \subset \mathfrak{g}$, denote by $\check{T}$ the corresponding Cartan subgroup (i.e. a maximal torus) in $G_r$. Let us choose a (positive) Weyl chamber $\mathfrak{t}_{+}^{*}$ in $\mathfrak{h}^{*} \subset \mathfrak{g}^{*}$. It turns out that the pre-image $\mu^{-1} (\mathfrak{t}_{+}^{*})$ is a ``symplectic cross-section'' (in the words of \cite{GuilleminSternberg}) of the $G_r$ action on $(M, \omega)$:

\begin{theo}\label{Guillemin-Sternberg}(Guillemin-Sternberg \cite{GuilleminSternberg}) The pre-image $\mu^{-1}(\mathfrak{t}_{+}^{*})$ of the positive Weyl chamber is a $\check{T}$-invariant symplectic submanifold of $(M, \omega)$. The restriction of the $G_{r}$ action on $\mu^{-1}(\mathfrak{t}_{+}^{*})$ is a Hamiltonian torus action of $\check{T}$. For any closed subgroup $\check{T}' \subset \check{T}$, the subset of $\mu^{-1}(\mathfrak{t}_{+}^{*})$ containing points fixed by $\check{T}'$ is a $\check{T}$ symplectic submanifold of $\mu^{-1}(W_{+})$.
\end{theo}

Since $G_r$ is a compact connected Lie group, the Cartan subalgebras in $\mathfrak{g}^{*}$ are conjugate to each other. As $\mu$ interwines the $G_r$ action on $(M, \omega)$ and the coadjoint action of $G_r$ on $\mathfrak{g}^{*}$, any two of these ``symplectic cross-sections'' is the image under the $G_r$-action of each other.

\begin{rem} The original statement also requires $M$ to be compact. Nevertheless, in order only to get the cited statements, the compactness is not necessary.
\end{rem}

In the three-body or N-body problems in $\R^{3}$, the group SO(3) acts in a Hamiltonian way on the (translation-reduced) phase space, whose moment map is just the angular momentum vector $\vec{C} \in \mathfrak{so}(3) \cong \R^{3}$. Any Cartan subalgebra is the vector space of infinitesimal generators of rotations with fixed rotation axis, which is a 1-dimensional vector subspace (homeomorphic to $\R$) in $\R^{3}$. A positive Weyl chamber is therefore a connected component of this 1-dimensional vector subspace minus the origin, formed by infinitesimal generators generating rotations with the same orientation. The pre-image of the positive Weyl chamber is the submanifold one gets by fixing the direction of $\vec{C}$, which is easily seen to be invariant under the Hamiltonian flow of the N-body problem. Theorem \ref{Guillemin-Sternberg} shows that this submanifold is symplectic and the restriction of the SO(3)-action to this submanifold is the SO(2)-action around the fixed direction of $\vec{C}$. This is exactly the ``partial reduction'' procedure described in \cite{Malige}. 

Moreover, as already mentioned in Subsection \ref{Subsection: Deprit 3BP}, Jacobi explicitly establishes a set of action-angle coordinates on $\Pi'$ from the Delaunay coordinates. We state a theorem in the next section, which allows us to easily deduce Deprit's coordinates from those found by Jacobi, construct the Deprit coordinates for more bodies, and prove the symplecticity of these coordinates.

\section{Symplectic Complement of the Symplectic Cross-Sections}\label{Sec 5}

\begin{theo}\label{Theo: Generalized Rotation Lemma} Suppose that a compact connected Lie group $G_r$ acts in a Hamiltonian way on the symplectic manifold $(M, \omega)$ with moment map $\mu: M \to \mathfrak{g}^{*}$. Let us fix a Weyl chamber $\mathfrak{t}_{+}^{*}$ in $\mathfrak{g}^{*}$. Suppose that $\forall x \in \mu^{-1}(\mathfrak{t}_{+}^{*})$, $\mu$ induces an isomorphism between the $\omega$-orthogonal space of $\mu^{-1}(\mathfrak{t}_{+}^{*})$ at $x$ and the tangent space at $\mu(x)$ of a coadjoint orbit in $\mathfrak{g}^{*}$, and suppose that, up to multiplication of a constant, there exists only one $G_r$-invariant symplectic form on the coadjoint orbit, so that there exists only one symplectic form on the normal space of $\mu^{-1}(\mathfrak{t}_{+}^{*})$ at $x$ which can be extended to a $G_r$-invariant form along a $G_r$-orbit. Then there exists a constant $D_{\mu_{0}} \in \R$, such that
$$\omega=\omega_{0} \, + \mu^{*} \tilde{\omega}_{\mu_{0}},$$
where $\omega_{0}$ is the restriction of $\omega$ on $\mu^{-1} (\mathfrak{t}_{+}^{*})$ and $D_{\mu_{0}} \, \mu^{*} \tilde{\omega}_{\mu_{0}}$ are seen as extended to two $G_r$-invariant two-forms, and for $x_{0} \in M$, $\omega_{\mu_{0}}$ is the Kirillov-Konstant symplectic form on the coadjoint orbit passing through $\mu_{0}=\mu(x_{0})$ and $D_{\mu_{0}}$ depends only on the coadjoint orbit of $\mu_{0}$.
 \end{theo}

\begin{proof} We fix a point $x_{0} \in \mu^{-1}(\mathfrak{t}_{+}^{*})$. For any two vectors $v_{1}, v_{2} \in T_{x_{0}} M$, we may decompose them as $v_{1}=u_{1}+w_{1}, v_{2}=u_{2}+w_{2}$, such that $u_{1}, u_{2} \in T_{x_{0}} \mu^{-1} (\mathfrak{t}_{+}^{*})$ and $w_{1}, w_{2} \in (T_{x_{0}} \mathfrak{t}_{+}^{*})^{\perp}$, in which $(T_{x_{0}} \mathfrak{t}_{+}^{*})^{\perp}$ is the orthogonal space of $T_{x}  \mu^{-1} (\mathfrak{t}_{+}^{*})$ with respect to the form $\omega$. The statement is equivalent to 
$$\forall w_{1}, w_{2} \in (T_{x_{0}} \mathfrak{t}_{+}^{*})^{\perp}, \omega_{x_{0}} (w_{1}, w_{2})=D \, \mu^{*} \tilde{\omega}_{\mu_{0}}(w_{1}, w_{2}).$$
Restricted to the $\omega$-orthogonal space of $\mu^{-1}(\mathfrak{t}_{+}^{*})$ at $x_{0}$, both forms $\omega$ and $\mu^{*} \tilde{\omega}_{\mu_{0}} $ are bilinear, anti-symmetric, non-degenerate, and they can be extended to two $G_r$-invariant forms. Therefore after being extended in such ways, they agree up to a $G_r$-invariant factor $D$ (and hence $D$ is constant on the pre-image of a coadjoint orbit). Hence we have 
$$\omega=\omega_{0}+D \, \mu^{*}(\, \tilde{\omega}_{\mu_{0}}),$$
in which $D=D_{\mu_{0}}$ depends only on the coadjoint orbit of $\mu_{0}$. 

\end{proof}

\section{Symplecticity of Delaunay and Deprit coordinates}\label{Symplecticity of Delaunay and Deprit coordinates}
In the spatial cases, the symmetric group is always SO(3). The SO(3) coadjoint orbits we shall consider are homeomorphic to $S^{2}$, which admit only one SO(3)-invariant symplectic form up to multiplication of constants. The SO(3)-moment map is the total angular momentum $\vec{C}$ and the form $\mu^{*} \tilde{\omega}_{\mu_{0}}$ is seen to be equal to $d C_{z} \wedge d h_{\vec{C}}$ (by passing to symplectic cylindral coordinates). Our main task in this section is to determine the factor $D$ in different contexts.

\subsection{Delaunay coordinates}
\subsubsection{Planar Delaunay coordinates}
We first analyze the planar Delaunay coordinates $(L, l, G, g)$. Let $K$ be the energy of the planar Kepler problem. When $K$ is negative, we know that all its orbits are closed. Consider two commuting SO(2)-actions on the phase space, one by shifting the phase along the elliptic orbits, and another one by rotating the orbits in the plane. 

\begin{claim} $G$ is the moment map associated to the Hamiltonian action of the group SO(2) acting by simultaneous rotations in positions and in momenta on the phase space. An SO(2)-orbit is parametrized by the argument of the perihelion $g$ (when this angle is well-defined).
\begin{proof} This is a standard calculation of a SO(2)-moment map.
\end{proof}
\end{claim}

\begin{claim} $L$ is the moment map associated to the Hamiltonian action of the group SO(2) on the phase space by phase shifts on the Keplerian elliptic orbits. An SO(2)-orbits is parametrized by the mean anomaly $l$.
\begin{proof} The second and third laws of Kepler implies that the moment map associated to this SO(2)-action is independent of the eccentricity of the elliptic orbit. It is thus enough to calculate this moment map along orbits with zero eccentricity, \emph{i.e.}, the circular orbits, along which the $SO(2)$-action is just simultaneous rotations in positions and momenta, and the moment map is easily seen to be the (circular) angular momentum $L$. 
\end{proof}
\end{claim}

It is direct to verify that $(L, l, G, g)$ are functionally independent. Moreover, similarly as in \cite{Fejoz2013}, in terms of the Poisson brackets, we have
\begin{itemize}
\item $\{L, l\}=\{G, g\}=1$, by definition of the moment map.
\item $\{L, g\}=\{G, l\}=0$, by definition of the moment map and the commutativity of the two SO(2)-actions.
\item $\{L, G\}=0$, by the fact that $G$ is a first integral for the Kepler problem.
\item $\{l, g\}=0$, as a result of the first three Poisson brackets, the symplectic form may only be written in the form $d L \wedge d l + d G \wedge d g + f \, d l \wedge d g$. by closedness of this 2-form, $f=f(l, g)$ depends only on $l, g$. as the SO(2)-action of the angle $l$ is Hamiltonian, the 1-form $d L + f d g$ must be exact, which implies $f=f(g)$ only depends on $g$. Let $f(g) d g= d F(g)$, then $L+F(g)$ is a moment map associated to $l$. As two SO(2)-moment maps may only differ by a constant, $F(g)$ must be a constant, which in turn implies that $f=0$.
\end{itemize}

\subsubsection{Spatial Delaunay coordinates}
Based on the symplecticity of the planar Delaunay coordinates, a direct application of Theorem \ref{Theo: Generalized Rotation Lemma} confirms the symplecticity of the spatial Delaunay coordinates up to an indetermined factor $D=D(G)$. To determine $D$, we go through a limiting procedure by letting the orbital plane tends to horizontal. Some care must be taken since  the angles $g$ and $h$ in the Delaunay variables are not well-defined for horizontal ellipses. We thus restrict to the submanifold for which all the spatial Delaunay variables are well defined and such that $g=0$, i.e. the direction of the perihelion of the ellipse agrees with the direction of the ascending node. The 2-form $d L \wedge d l + d G \wedge d g + D d H \wedge d h$ is thus restricted to $d L \wedge d l + D d H \wedge d h$ on this submanifold. Thanks to the restriction, the angle $h$ is now exactly the angle between the direction of the perihelion and the first coordinate axis, which remain well-defined when the orbital plane is horizontal. Thus the form $d L \wedge d l + D d H \wedge d h$ can be extended continuously (and actually smoothly) to horizontal orbital plane after the restriction. However, for horizontal ellipse, we have $H=G$, and the angle $h$ agrees with the planar argument of the perihelion (the angle $g$ in the planar Delaunay coordinates). By comparing with the planar Delaunay coordinates, we find $D=1$.

\subsection{Deprit coordinates for the three-body problem}
It is not hard to deduce from Jacobi's elimination of node that a set of Darboux coordinates on the partially reduced space is $(L_{1}, l_{1}, G_{1}, \bar{g}_{1}, L_{2}, l_{2}, G_{2}, \bar{g}_{2}, C, \phi_{1})$. Therefore, from Theorem \ref{Theo: Generalized Rotation Lemma} we know that the symplectic form on a dense open set of $\Pi$ takes the form
$d L_{1} \wedge d l_{1} + d G_{1} \wedge d \bar{g}_{1}+d L_{2} \wedge d l_{2} + d G_{2} \wedge d \bar{g}_{2}+ d C \wedge d \phi_{1} + D\, d C_{z} \wedge d \phi_{2}$. 
Now let us determine the factor $D=D(C)$ by some limiting procedures. We see that the term $D\, d C_{z} \wedge d \phi_{2}$ does not depend specifically on $L_{1}$. 

To determine the constant $D$, we restrict the form $d L_{1} \wedge d l_{1} + d G_{1} \wedge d \bar{g}_{1}+d L_{2} \wedge d l_{2} + d G_{2} \wedge d \bar{g}_{2}+ d C \wedge d \phi_{1} + D d C_{z} \wedge d \phi_{2}$ to the symplectic subspace of identical Keplerian motions, in which we have $L_{1}=L_{2}, l_{1}=l_{2}, G_{1}=G_{2}=\dfrac{C}{2}, \bar{g}_{1}+\phi_{1}=\bar{g}_{2}+\phi_{2}=g_{1}=g_{2}, H_{1}=H_{2}=\dfrac{C_{z}}{2}, \phi_{2}=h_{1}=h_{2}$. By comparing with the restriction of the (decoupled) Delaunay coordinates, we find $D=1$.

\subsection{Deprit coordinates for the $N$-Body problem}
The proof is inductive by taking the Deprit coordinates for the first $N-2$ Keplerian ellipses and Delaunay coordinates for the last Keplerian ellipse. A partial reduction procedure again gives symplectic coordinates on the invariant subspace obtained by fixing the direction of the total angular momentum. Theorem \ref{Theo: Generalized Rotation Lemma} thus confirms the desired result except for the determination of the constant $D$. We can now take the first $2$ Keplerian elliptic motions as identical (therefore we can consider one (fictitious) Keplerian elliptic motion with twice of the circular angular momentum and the angular momentum instead of them) and finish the argument by comparing the resulting coordinates with the Deprit coordinates for the first $N-2$ Keplerian ellipses. We find $D=1$. 

\begin{ack} Many thanks to Alain Chenciner and Jacques Féjoz for their constant help during years, and to them, together with Henk Broer, for comments and suggestions.
\end{ack}

\bibliography{QuasiperiodicMotionSpatial}

\end{document}